\documentclass[12 pt]{article}

\usepackage{graphicx}
\usepackage{amsmath}
\usepackage{amssymb}
\usepackage{amsthm}
\usepackage{pxfonts}
\usepackage{enumerate}
\usepackage{color}
\usepackage{mathdots}
\usepackage{sectsty}
\usepackage[hidelinks]{hyperref}
\usepackage{tikz}
\usepackage{caption}
\usepackage{adjustbox}

\sectionfont{\scshape\centering\fontsize{11}{14}\selectfont}
\subsectionfont{\scshape\fontsize{11}{14}\selectfont}
\usepackage{fancyhdr}

\newcommand\shorttitle{On asymptotics of partitions into polynomials}
\newcommand\authors{Nian Hong Zhou}

\hypersetup{
    colorlinks=true,       % false: boxed links; true: colored links
    linkcolor=blue,          % color of internal links (change box color with linkbordercolor)
    citecolor=cyan,        % color of links to bibliography
}
\usepackage{verbatim}
\usepackage{url}

\makeatother

\newtheorem{theorem}{Theorem}[section]
\newtheorem{lemma}{Lemma}[section]
\newtheorem{corollary}[theorem]{Corollary}
\newtheorem{proposition}[lemma]{Proposition}

\theoremstyle{remark}

\usepackage{tikz}
\usetikzlibrary{decorations.markings}
\makeatother

\def\ri{\mathrm i}

\def\qb{\mathbb Q}
\def\rb{\mathbb R}
\def\nb{\mathbb N}
\def\zb{\mathbb Z}

\def\cb{{\mathbb C}}

\def\rrw{\rightarrow}
\numberwithin{equation}{section}

\title{\large \bf Note on partitions into polynomials with number of parts in an arithmetic progression}
\author{\small Nian Hong Zhou}

\date{} % Date for the report

\begin{document}
\maketitle
\begin{abstract}
Let $f: \zb_+\rrw \zb_+$ be a polynomial with the property that corresponding to
every prime $p$ there exists an integer $\ell$ such that $p\nmid f(\ell)$. In this paper, we establish some equidistributed results between the number of partitions of an integer $n$ whose parts are taken from the sequence $\{f(\ell)\}_{\ell=1}^{\infty}$ and the number of parts of those partitions which are in a certain arithmetic progression.

%We prove that there exist a constant $\delta_{f}>0$ depending only on $f$ such that
%$$p_f(a,k;n)=\frac{p_f(n)}{k}\left(1+O\left(n\exp\left(-\delta_{f}k^{-2}n^{\frac{1}{1+\deg(f)}}\right)\right)\right),$$
%holds uniformly for all $a,k, n\in\zb_+$ with $k^{2+2\deg(f)}\ll n$,  as $n$ tending to infinity.
\end{abstract}

%\keywords{Asymptotics, Partitions, Crank, Rank.}
%\subjclass[2010]{Primary: 11P82; Secondary:  05A16, 05A17.}
\maketitle

%\tableofcontents
\section{Introduction and statement of results}

\subsection{Background}\label{sec11}
We begin with some standard definitions from the theory of partitions \cite{MR0557013}.
A \emph{partition} is a finite non-increasing sequence $\pi_1,\pi_2\dots,\pi_m$ for some integer $m\ge 1$ such that each $\pi_j$ is a positive integer. The $\pi_j$ are called the \emph{parts} and $m$ is called the number of parts of the partition. We say $\pi$ is a partition of $n$ if $\pi_1+\pi_2+\dots+\pi_m=n$.\newline

In this paper, if not specially specified, $f(x)\in\qb[x]$ is an integer-valued polynomial which has the property that corresponding to every prime $p$ there exists an integer $\ell$ such that $p\nmid f(\ell)$ and $f(\ell)>0$ for all integers $\ell\ge 1$. Clearly, polynomials
$$x^{r},\; ax+b,\; x(x+1)/2,\; cx(x+1)(x+2)+1,\ldots,$$
with $r,a,b,c\in\zb_+$ and $\gcd(a,b)=1$ satisfy the conditions on $f$ above.  Let $p_f(n)$ denote the number of partitions of $n$ whose parts are taken from the sequence $\{f(\ell)\}_{\ell=1}^{\infty}$.
%i.e.,
%$$p_f(n)=\#\{\pi\in\mathcal{P}_f: \sigma(\pi)=n\},$$
Then by Andrews \cite[Theorem 1.1]{MR0557013},
\begin{equation}\label{gf}
G_f(z):=\sum_{n\ge 0}p_f(n)q^{n}=\prod_{n\ge 1 }\frac{1}{1-q^{f(n)}}.
\end{equation}
Here and throughout this section, $q=e^{-z}, z\in\cb$ with $\Re(z)>0$.  Let $p_{f}(m,n)$ denote the number of partitions of $n$ whose parts lie in the sequence $\{f(\ell)\}_{\ell=1}^{\infty}$ and with exactly $m$ parts. Also, by Andrews \cite[p. 16]{MR0557013} we have
\begin{equation*}
G_f(\zeta, q):=\sum_{m,n\ge 0}p_f(m,n)\zeta^{m}q^{n}=\prod_{n\ge 1}\frac{1}{1-\zeta q^{f(n)}},
\end{equation*}
where $\zeta\in\cb$ and $|\zeta|<1/|q|$. Furthermore, letting $k\in\zb_+$ and $a\in\zb$ and denoting by $p_{f}(a,k;n)$ the number of partitions of $n$ whose parts are taken from the sequence $\{f(\ell)\}_{\ell=1}^{\infty}$, with the number of parts of those partitions congruent to $a$ modulo $k$, we have
\begin{equation}\label{eq1}
p_{f}(a,k;n)=\sum_{\substack{m\ge 0\\ m\equiv a\pmod {k}}}p_f(m,n).
\end{equation}
%If $r=1$ then by $q$-binomial theorem,
%$$\prod_{n\ge 1}\frac{1}{1-xq^n}=\sum_{\ell\ge 0}\frac{q^{\ell}}{(q;q)_{\ell}}x^{\ell},$$
%we have
%\begin{equation}
%\sum_{n\ge 0}p_{1}(k,a;n)q^n=\sum_{\ell\ge 0}\frac{q^{\ell k+a}}{(q;q)_{\ell k+a}}.
%\end{equation}

Determining the values of $p_f(n)$ has a long history and can be traced back to the work of Euler. The most famous example is when $f(n)=n$, which corresponds to unrestricted integer partitions. In this case $p_f(n)$ is usually denoted by $p(n)$. Hardy and Ramanujan \cite{MR1575586} proved
$$p(n)\sim \frac{1}{4\sqrt{3}n}e^{2\pi\sqrt{n/6}},$$
as $n\rrw \infty$. Let $f_r(n)=n^r$ with $r\in\zb_+$; we then obtain the $r$-th power partition function $p_r(n):=p_{f_r}(n)$. Hardy and Ramanujan \cite[p.~111]{MR1575586} conjectured that
\begin{equation}\label{eq00}
p_{r}(n)\sim \frac{c_rn^{\frac{1}{r+1}-\frac{3}{2}}}{\sqrt{(2\pi)^{1+r}(1+1/r)}}e^{(r+1)c_rn^{\frac{1}{r+1}}},
\end{equation}
as $n\rrw \infty$, where $c_r\in\rb_+$ is a constant. The conjecture \eqref{eq00} has been proven by Wright \cite[Theorem~2]{MR1555393}. An asymptotic expansion for $p_f(n)$ has been established in the last paragraph of Roth and Szekeres \cite[p. 258]{MR67913}~(see also \cite[Equ. (1*), p. 258]{MR67913}).

After Hardy and Ramanujan and after Wright, such type problems have been widely investigated in many works in the literature. We shall refer the reader to Rademacher \cite{MR1575213}, Ingham \cite{MR5522}, Roth and Szekeres \cite{MR67913},  Meinardus \cite{MR62781} and Richmond \cite{MR0382210} for examples.

%Such type problems has attracted wide attention of many authors. For the cases of $\deg(f)=1$, Rademacher \cite{MR1575213}, Lehner \cite{MR0005523}, Livingood \cite{MR0012101}, Petersson \cite{MR0071458, MR0077566}, Iseki \cite{MR0123551, MR0108473} and many others, has obtained exact convergent series for certain unrestrict or restrict partition functions $p_f(n)$.
%For the cases of $\deg(f)\ge 2$, some new asymptotic expansions for $p_{r}(n)$ have recently
%since Wright¡¯s proof of $p_{r}(n)$ relied heavily on a transformation for the generating function $G_{f_r}(z)$ that involved generalised Bessel functions,
%established in Vaughan \cite{MR3376217} and Gafni \cite{MR3459558} by using Hardy--Littlewood circle method, and in Tenenbaum, Wu and Li \cite{MR3991428} by using saddle-point method. Berndt, Malik and Zaharescu \cite{MR3856854} have derived an asymptotic formula for $p_f(n)$ with $f(n)=(an-b)^r, a,b\in\zb_+, 1\le b<a$ and $\gcd(a,b)=1$, that is the restricted partitions in which each part is a $r$-th power in an arithmetic progression. Dunn and Robles \cite{MR3730445} have derived an asymptotic formula for $p_f(n)$ when $f$ satisfies certain mild conditions. We note that both \cite{MR3856854} and \cite{MR3730445} use the Hardy--Littlewood circle method.

\subsection{Main results}
In this paper, we investigate equidistributed properties of $p_{f}(a,k;n)$. To give our main results we need the following quantity $\Pi_f$ defined as
\begin{equation}
\Pi_f=\prod_{\substack{p~prime,~ s\in\zb_+ \\ p^s \parallel (f(\ell)-f(0)),\;\forall \ell\in\zb}}p^s,
\end{equation}
where $p^s\parallel m$ means that $p^s$ divides $m$ and $p^{s+1}$ not divide $m$.

The first result of this paper is stated as in the following.
\begin{proposition}\label{main1}Let $a, k\in \zb_+$ and $\delta \mid \Pi_f$. Then we have $p_f(a,\delta k; n)=0$ for $\delta\nmid (n-af(0))$ and
$$
\sum_{\substack{n\ge 0\\ n\equiv af(0)\pmod \delta}}\left(p_f(a,\delta k;n)-\frac{p_f( n)}{k}\right)q^n=\frac{1}{k\delta}\sum_{\substack{1\le j<k\\ 0\le \ell<\delta}}\zeta_{\delta k}^{-ja-k\ell af(0)}G_f\left(\zeta_{\delta k}^{j}, \zeta_{\delta }^{\ell} q\right),
$$
where $\zeta_{a}^b:=e^{{2\pi\ri b}/{a}}$. In particular, $p_f(a,\delta ;n)=p_f(n)$ for  $\delta\mid (n-af(0))$.
\end{proposition}
From Proposition \ref{main2}, we further prove
\begin{theorem}\label{main2}Let $n\in\zb_+$. For $\delta \mid \Pi_f$, $a, k\in \zb_+$ such that for any prime $p|k$ one has $(p\delta)\nmid \Pi_f$ and $k=o\left(n^{1/(2+2\deg(f))}/\sqrt{\log n}\right)$, there exists a constant $\delta_{f}\in\rb_+$ depending only on $f$ such that
$$p_f(a,\delta  k;n)=\frac{p_f(a,\delta; n)}{k}\left(1+O \left(n\exp\left(-\delta_{f}k^{-2}n^{\frac{1}{1+\deg(f)}}\right)\right)\right),$$
as $n\rrw \infty$ with $n\equiv af(0)~(\bmod \delta)$.
\end{theorem}

The above results immediately give the following corollary.
\begin{corollary}\label{cor1}For $a, k, n\in \zb_+$ such that $\gcd(k,\Pi_f)=1$,
$$p_f(a, k;n)\sim k^{-1}p_f(n),$$
as $n\rrw \infty$, holds for $k=o\left(n^{\frac{1}{2+2\deg(f)}}(\log n)^{-\frac{1}{2}}\right)$ .
\end{corollary}

We note that the case of $f(x)=x^r~(r\ge 2)$, $k=2$ of above Corollary \ref{cor1} was conjectured by Bringmann and Mahlburg \cite{BM} in their unpublished notes, which was proven by Ciolan \cite{Ciolan, Ciolan1} recently, by using a more complicate method.

\paragraph{Notations.}
The symbols $\zb$, $\mathbb{Z}_+$, $\mathbb{R}$ and $\mathbb{R_+}$ denote the set of
the integers, the positive integers, the real numbers and the positive real numbers, respectively. $e(z):=e^{2\pi i z}$, $\zeta_{a}^b:=e^{2\pi\ri b/a}$ and $\|x\|:=\min_{y\in\zb}|y-x|$.  If not specially specified, all the implied constants of this paper in $O$ and $\ll$ depend only on $f$.
\paragraph{Acknowledgements.}The author would like to thank the referee for very helpful and detailed comments and suggestions. This research was partly supported by the National Science Foundation of China (Grant No. 11971173).
\section{The proof of the main results of this paper}
\subsection{The proof of Proposition \ref{main1}}\label{sec21}
By the definition of $\Pi_f$, if $\delta |\Pi_f$ then for any $\ell\in\zb$, $\delta |(f(\ell)-f(0))$. By the condition on $f$, for every prime $p|\delta$, there exists an $\ell_0\in\zb$ such that $p\nmid f(\ell_0)$, and hence $p\nmid f(0)$. This mans that $\gcd(f(0), \delta)=1$. Therefore, there exists an integer $\hat{f}_\delta\in \zb$ such that
\begin{equation}\label{eqfell}
\hat{f}_{\delta}f(\ell)\equiv\hat{f}_{\delta}f(0)\equiv 1~(\bmod~ \delta)
\end{equation}
holds for all $\ell\in\zb$. Therefore, for any $k\in\zb_+$, using the orthogonality of roots of unity
$$\frac{1}{\delta k}\sum_{0\le j<\delta k}\zeta_{\delta k}^{j(m-a)}
=\begin{cases}1\quad &m\equiv a \pmod {\delta k},\\
0 &m\not\equiv a \pmod {\delta k},
\end{cases}
$$
we have
\begin{align*}
\sum_{n\ge 0}p_f(a,\delta k;n)q^{n}
&=\sum_{n\ge 0}q^n\sum_{m\ge 0}p_f(m,n)\frac{1}{\delta k}\sum_{0\le j<\delta k}\zeta_{\delta k}^{j(m-a)}\\
&=\frac{1}{k\delta}\sum_{0\le j<k\delta}\zeta_{\delta k}^{-ja}\prod_{n\ge 1}\frac{1}{1-\zeta_{\delta k}^{j} q^{f(n)}}\\
&=\frac{1}{k\delta}\sum_{0\le j<k}\zeta_{\delta k}^{-ja}\sum_{0\le \ell <\delta}\zeta_{\delta }^{-\ell a}\prod_{n\ge 1}\frac{1}{1-\zeta_{\delta k}^{j}\zeta_{\delta }^{\ell}q^{f(n)}}.
\end{align*}
Inserting \eqref{eqfell} into the above we obtain
\begin{align*}
\sum_{n\ge 0}p_f(a,\delta k;n)q^{n}&=\frac{1}{k\delta}\sum_{0\le j<k}\zeta_{\delta k}^{-ja}\sum_{0\le \ell <\delta}\zeta_{\delta }^{-\ell a}\prod_{n\ge 1}\frac{1}{1-\zeta_{\delta k}^{j}\zeta_{\delta }^{\ell \hat{f}_{\delta}f(n)}q^{f(n)}}\\
&=\frac{1}{k\delta}\sum_{0\le j<k}\zeta_{\delta k}^{-ja}\sum_{0\le \ell <\delta}\zeta_{\delta }^{-\ell a}G_f\left(\zeta_{\delta k}^{j}, \zeta_{\delta }^{\hat{f}_{\delta}\ell} q\right).
\end{align*}
Since $\gcd(\delta,\hat{f}_{\delta}) = 1$, the map $\ell\mapsto \hat{f}_{\delta}\ell$ permutes the residues modulo $\delta$. Thus by noting that $\hat{f}_{\delta}f(0)\equiv 1\pmod \delta$ we have
\begin{align*}
\sum_{n\ge 0}p_f(a,\delta k;n)q^{n}
&=\frac{1}{k\delta}\sum_{0\le j<k}\zeta_{\delta k}^{-ja}\sum_{0\le \ell <\delta}\zeta_{\delta }^{-\ell af(0)}G_f\left(\zeta_{\delta k}^{j}, \zeta_{\delta }^{\ell} q\right).
\end{align*}
This completes the proof by using the orthogonality of roots of unity.

\subsection{The proof of Theorem \ref{main2}}

To prove Theorem \ref{main1}, we need the following leading asymptotics of $p_f(n)$, which is deduced from Roth and Szekeres \cite[Equation (1*), p. 258]{MR67913}.
%\footnote{We note that the computation of asymptotic formula for $p_f(n)$ (the notation here is $p_u^*(n)$) of the last paragraph of \cite[p. 258]{MR67913} is wrong.}.
%For the $f$ satisfies the mild hypotheses of \cite{MR3730445}, the asymptotics for $p_f(n)$ follows from \cite[Theorem 1.1]{MR3730445}. However, our hypotheses on $f$ is more mild then \cite{MR3730445}. Thus
\begin{proposition}\label{pro1}We have
\begin{equation*}
p_f(n)\sim \frac{G_f(x)e^{nx}}{\sqrt{2\pi A_2(n)}},
\end{equation*}
as $n\rrw +\infty$, with $x, A_2(n)\in\rb_+$ given by
\begin{equation*}
n=\sum_{\ell\ge 1}\frac{f(\ell)}{e^{f(\ell)x}-1}\; \text{and}\; A_2(n)=\sum_{\ell\ge 1}\frac{f(\ell)^2e^{f(\ell)x}}{(e^{f(\ell)x}-1)^2}.
\end{equation*}
\end{proposition}

To obtain the proof of our main result, we compute the leading asymptotics of $x$ and $A_2(n)$ as $n\rrw\infty$. Using integration by parts for a Riemann--Stieltjes integration,
\begin{align}\label{eqm23}
\frac{\,d^j}{\,d x^j}\log G_f(x)=&-\frac{\,d^j}{\,d x^j}\sum_{n\in\zb_+}\log(1-e^{-f(n)x})\nonumber\\
=&-\int_{1^-}^{\infty}\frac{\,d^j}{\,d x^j}\log(1-e^{-tx})\,d\bigg(\sum_{\substack{n\in\zb_+,~ f(n)\le t}}1\bigg).
\end{align}
On the other hand, for any $t\in \rb_+$,
\begin{align}\label{eqm24}
\sum_{\substack{n\in\zb_+,~ f(n)\le t}}1&=\sum_{\substack{1\le n\le t^{1/(2r)},~ f(n)\le t}}1+\sum_{\substack{n> t^{1/(2r)},~ f(n)\le t}}1\nonumber\\
&=O(t^{1/(2r)})+\sum_{\substack{n> t^{1/(2r)}\\ a_r^{1/r}n(1+O(n^{-1})\le t^{1/r}}}1\nonumber\\
&=O(t^{1/(2r)})+(t/a_r)^{1/r}\left(1+O(t^{-1/(2r)})\right)\nonumber\\
&=(t/a_r)^{1/r}+O(t^{1/(2r)}).
\end{align}
Inserting \eqref{eqm24} into \eqref{eqm23} we obtain
\begin{align*}
\frac{\,d^j}{\,d x^j}\log G_f(x)
=&-\frac{1}{a_{r}^{1/r}}\int_{1}^{\infty}\frac{\,d^j}{\,d x^j}\log(1-e^{-u^rx})\,du\\
&+O_f\left(|\log x|+\frac{1}{x^j}+\int_{1}^{\infty}t^{1/(2r)}\left|\frac{\,d}{\,dt}\frac{\,d^j}{\,d x^j}\log(1-e^{-tx})\right|\,dt\right)\\
=&-\frac{1}{a_{r}^{1/r}}\int_{0}^{\infty}\frac{\,d^j}{\,d x^j}\log(1-e^{-u^rx})\,du+O\left(x^{-j-1/(2r)}\right).
\end{align*}
Notice that
\begin{align*}
-\int_{0}^{\infty}\log(1-e^{-u^rx})\,du&=\sum_{\ell\ge 1}\int_{0}^{\infty}\frac{e^{-u^r\ell x}}{\ell}\,du\\
&=\sum_{\ell\ge 1}\frac{1}{\ell^{1+1/r}x^{1/r}}\int_{0}^{\infty}e^{-u^r}\,du
=\frac{\zeta(1+1/r)\Gamma(1+1/r)}{x^{1/r}}.
\end{align*}
Then for each $j\in\zb_{\ge 0}$,
$$
\frac{\,d^j}{\,d x^j} \log G_f(x)=\frac{\,d^j}{\,d x^j}\frac{\zeta(1+1/r)\Gamma(1+1/r)}{(a_{r} x)^{1/r}}+O\left(\frac{1}{x^{j+1/(2r)}}\right),
$$
as $x\rrw 0^+$. This means that
$$n=-\frac{\,d}{\,d x} \log G_f(x)=\frac{\zeta(1+1/r)\Gamma(1+1/r)}{ra_{r}^{1/r}x^{1+1/r}}+O\left(\frac{1}{x^{1+1/(2r)}}\right)$$
and
$$A_2(n)=\frac{\,d^2}{\,d x^2} \log G_f(x)=\frac{\zeta(1+1/r)\Gamma(1+1/r)}{r(1+1/r)a_{r}^{1/r}x^{2+1/r}}+O\left(\frac{1}{x^{2+1/(2r)}}\right).$$

Therefore using Proposition \ref{pro1} and the above we find that there exist constants $c_1(f), c_2(f)\in\rb_+$ depending only on $f$ such that
\begin{equation}\label{eqasp}
p_f(n)\sim c_2(f) n^{-\frac{1+2\deg(f)}{2+2\deg(f)}}G_f(x)e^{nx},
\end{equation}
as $n\rrw \infty$, with $x\in\rb_+$ given by
\begin{equation}\label{eq210}
n=\sum_{\ell\in\nb}\frac{f(\ell)}{e^{f(\ell)x}-1}\sim c_1(f)x^{-1-1/\deg(f)}.
\end{equation}

We next prove the following mean square estimation for the difference between $p_{f}(a,k;n)$ and $k^{-1}p_f(n)$.
\begin{proposition}\label{pro2}
For $a, k\in \zb_+$, $k=O\left(n^{1/(2+2\deg(f))}\right)$ and $\delta \mid \Pi_f$ such that for any $p|k$ one has $(p\delta)\nmid \Pi_f$, there exists a constant $\delta_f'\in\rb_+$ depending only on $f$ such that
\begin{equation*}
\sum_{\substack{n\ge 0\\ n\equiv af(0)\pmod {\delta}}}\left|p_{f}(a,\delta k;n)-k^{-1}p_f(n)\right|^2e^{-2nx}\ll G_f(x)^{2}\exp\left(-2\delta_f'k^{-2}x^{-1/\deg(f)}\right),
\end{equation*}
as $x\rrw 0^+$.
\end{proposition}
Then, Theorem \ref{main2} follows from Proposition \ref{pro1} and Proposition \ref{pro2}, immediately. In fact, if we pick $n\equiv af(0)~(\bmod \delta)$, then by setting $x\sim (c_1(f)/n)^{\deg(f)/(1+\deg(f))}$ given by \eqref{eq210}, using \eqref{eqasp} and Proposition \ref{pro2}  for all positive integers $k$ such that $k^2\ll n^{1/(1+\deg(f))}$ implies that
\begin{align*}
p_{f}(a,\delta k;n)-\frac{p_f(n)}{k}&\ll e^{nx}\left(\sum_{\substack{j\ge 0\\ j\equiv af(0)\pmod {\delta}}}\left|p_{f}(a,\delta k;j)-\frac{p_f(j)}{k}\right|^2e^{-2jx}\right)^{1/2}\\
&\ll e^{nx}G_f(x)\exp(-\delta_f'k^{-2}x^{-1/\deg(f)})\\
&\ll \frac{p_f(n)}{k}\left(kA(n)^{1/2}\exp(-\delta_f'k^{-2}x^{-1/\deg(f)})\right)\\
&\ll \frac{p_f(n)}{k}n\exp\left(-{\delta_f}{k^{-2}}n^{\frac{1}{1+\deg(f)}}\right)
\end{align*}
holds for some constant $\delta_{f}\in\rb_+$ depending only on $f$, as the integer $n\rrw\infty$.

\section{The proof of Proposition \ref{pro2}}
In this section we prove Proposition \ref{pro2}. We shall always set $r=\deg(f)\in\zb_+$.
We first prove the following Lemma \ref{lem43} and Lemma \ref{lem41}.
\begin{lemma}\label{lem43}For all $a,b\in\zb_+$ with $1\le b<a$ and $y\in\rb$,
$$
\int_{0}^L\sin^2\left(\pi\left(\frac{b}{a}-f(u)y\right)\right)\,du\gg \frac{L}{a^2},
$$
as $L\rrw+\infty$.
\end{lemma}
\begin{proof}We prove the case of $y\ge 0$, and the case of $y\le 0$ is similar. For each positive $L$ sufficiently large, we estimate that
\begin{align*}
\int_{0}^L\sin^2\left(\pi\left(\frac{b}{a}-f(u)y\right)\right)\,du&\ge \int_{\substack{0\le u\le L\\ \left\| {b}/{a}-f(u)y\right\|> {1}/{(3a)}}}\sin^2\left(\frac{\pi}{3a}\right)\,du\\
&\gg \frac{1}{a^2}\int_{\substack{0\le u\le L\\ \left\| {b}/{a}-f(u)y\right\|> {1}/{(3a)}}}\,du.
\end{align*}
Clearly, $0\le |f(u)|\le f(X)$ holds for all $u\in[0,X]$ when $X$ is sufficiently large. Thus for $y\ge 0$ such that $0\le yf(L/12)\le {1}/{(2^ra)}$,
\begin{align*}
\int_{\substack{0\le u\le L\\ \left\| {b}/{a}-f(u)y\right\|> {1}/{(3a)}}}\,du\ge \int_{\substack{0\le u\le L/12\\ \left\| {b}/{a}-f(u)y\right\|> {1}/{(3a)}}}\,du=\frac{L}{12}.
\end{align*}
For $y\ge 0$ such that $yf(L/12)\ge {1}/{(2^ra)}$, it is not difficult to see that
\begin{align*}
\int_{\substack{0\le u\le L\\ \left\| {b}/{a}-f(u)y\right\|> {1}/{(3a)}}}\,du
&\ge  \int_{\substack{L/6\le u\le L\\ \left\| {b}/{a}-f(u)y\right\|> {1}/{(3a)}}}\,du\\
&\ge \sum_{\substack{\ell\in\zb_+,\; \ell\not\equiv b \pmod a\\ yf(\frac{L}{3})\le\frac{\ell}{a}\le yf(\frac{2L}{3})}}\int_{|f(u)y-\frac{\ell}{a
}|\le \frac{1}{2a}}\frac{1}{yf'(u)}\,d\left(yf(u)+\frac{\ell}{a}\right)
\end{align*}
holds for all sufficiently large $L$.
Notice that $yf'(u)\sim \frac{r}{u}yf(u)$ if $u\rrw \infty$. Then we have
\begin{align*}
\int_{\substack{0\le u\le L\\ \left\| {b}/{a}-f(u)y\right\|> {1}/{(3a)}}}\,du
&\gg  \sum_{\substack{\ell\in\zb_+,\;  \ell\not\equiv b \pmod a\\ ayf(\frac{L}{3})\le\ell\le ayf(\frac{2L}{3})}}\int_{|f(u)y-\frac{\ell}{a
}|\le \frac{1}{2a}}\frac{1}{(1/L)(\ell/a)}\,d\left(yf(u)-\frac{\ell}{a}\right)\\
&= L\sum_{\substack{\ell\in\zb_+,\;  \ell\not\equiv b \pmod a\\ ayf(\frac{L}{3})\le \ell\le ayf(\frac{2L}{3})}}
\frac{a}{\ell}\frac{1}{a}\gg L\sum_{\substack{\ell\in\zb_+,\; \ell\not\equiv b \pmod a\\ ayf(\frac{L}{3})\le \ell\le 2^r(1+O(1/L))ayf(\frac{L}{3})}}\frac{1}{\ell}.
\end{align*}
Since $A:=ayf(L/3)=4^r(1+O(1/L))ayf(L/12)\ge 2^r(1+O(1/L))$, it is not difficult to prove that
$$\sum_{\substack{\ell\in\zb_+,\;  \ell\not\equiv b \pmod a\\ A\le \ell\le 2^r(1+O(1/L))A}}\frac{1}{\ell}\gg 1$$
holds uniformly for all $A\ge 1$, as $L\rrw+\infty$. Therefore,
\begin{align*}
\int_{\substack{0\le u\le L\\ \left\| {b}/{a}-f(u)y\right\|> {1}/{(3a)}}}\,du\gg L
\end{align*}
holds for all $y\ge 0$ such that $yf(L/12)\ge {1}/{(2^ra)}$. This completes the proof.
\end{proof}

\begin{lemma}\label{lem41}For each integer $h\ge 2$ and each integer $d$ such that $\gcd(h,d)=1$ and $h\nmid\Pi_f$, we have
\begin{align*}
\left|\frac{1}{h}\sum_{1\le j\le h}e\left(f(j)\frac{d}{h}\right)\right|^2\le 1-\frac{4}{h^2}\sin^2\left(\frac{\pi}{h}\right).
\end{align*}
\end{lemma}
\begin{proof}We compute that
\begin{align*}
\left|\sum_{1\le j\le h}e\left(f(j)\frac{d}{h}\right)\right|^2=&\left|\sum_{1\le j\le h}e\left((f(j)-f(h))\frac{d}{h}\right)\right|^2\\
=&1+2\sum_{1\le j< h}\cos\left(\frac{2\pi (f(j)-f(h))d}{h}\right)+\left|\sum_{1\le j< h}e\left(\frac{(f(j)-f(h))d}{h}\right)\right|^2\\
\le& 1+2\sum_{1\le j< h}\cos\left(2\pi (f(j)-f(h))\frac{d}{h}\right)+(h-1)^2\\
=&h^2-4\sum_{1\le j< h}\sin^2\left(\pi (f(j)-f(0))\frac{d}{h}\right).
\end{align*}
From the definition of $\Pi_f$ and $h\nmid \Pi_f$, there exists an integer $j_0\in(0,h)$ such that $h\nmid(f(j_0)-f(0))$. This means that
\begin{align*}
\left|\frac{1}{h}\sum_{1\le j\le h}e\left(f(j)\frac{d}{h}\right)\right|^2\le 1-\frac{4}{h^2}\sin^2\left(\pi (f(j_0)-f(0))\frac{d}{h}\right)\le 1-\frac{4}{h^2}\sin^2\left(\frac{\pi}{h}\right),
\end{align*}
by recalling that $\gcd(h,d)=1$. This completes the proof.
\end{proof}
By the well-known Weyl's inequality, we prove
\begin{lemma}\label{lem42}Let $y\in\rb$, $h,d\in\zb$ with $h\ge 1$ and $\gcd(h, d)=1$ such that
$$\left|y-\frac{d}{h}\right|<\frac{1}{h^2}.$$
We have there exists a constant $C_f\in\zb_+$ depending only on $f$ such that if $h>C_f$ then
$$
\left|\sum_{1\le n\le L}e(f(n) y)\right|\le L^{1-2^{-r-1}}+Lh^{-2^{-r-1}}.
$$
\end{lemma}
\begin{proof}Set $f(n)=({b}/{a})n^r+a_{r-1}n^{r-1}+\ldots$ with $a,b\in\zb_+$ and $\gcd(a,b)=1$. By Weyl's inequality (see \cite[Lemma 20.3]{MR2061214}),
since $f(an+j)=ba^{r-1}n^r+(rba^{r-2}j+a_{r-1}a^{r-1})n^{r-1}+\ldots$ and $|y-d/h|\le 1/h^2$, we have
\begin{align*}
\sum_{1\le n\le L}e(f(n) y)&=\sum_{1\le j\le a}\sum_{\substack{n\ge 0\\ 1\le an+j\le L}}e(f(an+j) y)\nonumber\\
&\ll_{\varepsilon} \sum_{1\le j\le a} [(L-j)/a]^{1+\varepsilon}\left(h^{-1}+[(L-j)/a]^{-1}+h[(L-j)/a]^{-r}\right)^{2^{1-r}}.
\end{align*}
Therefore,
\begin{align}\label{eq40}
\sum_{1\le n\le L}e(f(n) y)\ll_{\varepsilon} L^{1+\varepsilon}(h^{-1}+L^{-1}+hL^{-r})^{2^{1-r}}\ll L^{1-2^{-r-1/2}}
\end{align}
holds for all integer $h\in( L^{1/2}, L^{r-1}]$. Also, by \cite[Corollary 20.4]{MR2061214}
\begin{equation}\label{eq36}
\sum_{1\le j\le h}e\left(f(j)\frac{d}{h}\right)\ll_{\varepsilon} h^{1-2^{1-r}+\varepsilon}\ll h^{1-2^{-r}}
\end{equation}
holds for all positive integers $h$.
On the other hand, by \cite[Equation 20.32]{MR2061214} we have
\begin{equation}\label{eq51}
\sum_{1\le n\le L}e\left(f(n)y\right)=\frac{1}{h}\sum_{1\le j\le h}e\left(f(j)\frac{d}{h}\right)\int_{0}^{L}e\left(f(u)\left(y-\frac{d}{h}\right)\right)\,du+O(h).
\end{equation}
Combining \eqref{eq36} we obtain
\begin{align*}
\sum_{1\le n\le L}e\left(f(n)y\right)\ll  h^{-2^{-r}}L+h\ll h^{-2^{-r}}L
\end{align*}
for all positive integers $h\le L^{1/2}$. Thus by the above and \eqref{eq40} we have there exists a constant $C_f\in\zb_+$ depending only on $f$ such that if $h>C_f$ then
$$\left|\sum_{1\le n\le L}e\left(f(n)y\right)\right|\le h^{-2^{-r-1}}L+L^{1-2^{-r-1}}.$$
This completes the proof of the lemma.
\end{proof}

From Lemma \ref{lem43}--Lemma \ref{lem42} we have
\begin{lemma}\label{Z}Let $\delta, k,j,L\in\zb_+$ and $\ell\in\zb$,  with $1\le j<k$ and $\delta|\Pi_f$. Further suppose that for any prime $p|k$, $(p\delta)\nmid \Pi_f$.  We have
$${\rm Re}\sum_{1\le n\le L}\left[1-e\left(\frac{k\ell+j}{\delta k}-f(n)y\right)\right]\gg k^{-2} L$$
uniformly for all real numbers $y$, as $k^{-2}L \rrw +\infty$.
\end{lemma}

\begin{proof}
By the well-known Dirichlet's approximation theorem, for any real number $y$ and positive integer $L$ sufficiently large, there exist integers $d$ and $h$ with $0<h\le L^{r-1}$ and $\gcd(h, d)=1$ such that
\begin{equation}\label{eq5}
\left|y-\frac{d}{h}\right|<\frac{1}{hL^{r-1}}.
\end{equation}
We prove the lemma by considering the following two cases.

For any real number $y$ satisfy the approximation \eqref{eq5} with $h|\Pi_f$, using the definition of $\Pi_f$ we have
\begin{align}\label{eq400}
{\rm Re}\sum_{1\le n\le L}&\left[1-e\left(\frac{k\ell+j}{\delta k}-f(n)y\right)\right]\nonumber\\
&=L-\Re\left(e\left(\frac{k\ell+j}{\delta k}-\frac{df(0)}{h}\right)\sum_{1\le n\le L}e\left(\frac{d(f(0)-f(n))}{h}+f(n)\left(\frac{d}{h}-y\right)\right)\right)\nonumber\\
&=L-\Re\left(e\left(\frac{k\ell+j}{\delta k}-\frac{df(0)}{h}\right)\left(\int_{0}^Le\left(f(u)\vartheta\right)\,du+O(1)\right)\right)\nonumber\\
&=2\int_{0}^L\sin^2\left(\pi\left(\frac{k\ell+j}{\delta k}-\frac{df(0)}{h}-f(u)\vartheta\right)\right)\,du+O(1),
\end{align}
 by using \eqref{eq51}, where $\vartheta=y-d/h$. Since $1\le j<k$, there exist $a\in\zb_+,b\in\zb$ such that $\gcd(a,b)=1$ and
$$\frac{b}{a}:=\frac{k\ell+j}{\delta k}-\frac{df(0)}{h}\in \left(\frac{1}{k\Pi_f}\zb\right)\setminus\zb.$$
Thus Lemma \ref{lem43} implies that
\begin{align}\label{eq400}
{\rm Re}\sum_{1\le n\le L}&\left[1-e\left(\frac{k\ell+j}{\delta k}-f(n)y\right)\right]\gg a^{-2}L\gg k^{-2}L
\end{align}
holds for all $k,L$ such that $k^{-2}L\rrw \infty$.

For any real number $y$ satisfying the approximation \eqref{eq5} with $h\ge 2$ and $h\nmid \Pi_f$, using \eqref{eq51} and Lemma \ref{lem41} we have
\begin{equation}\label{eq411}
\left|\sum_{1\le n\le L}e\left(f(n)y\right)\right|\le \left|\frac{1}{h}\sum_{1\le j\le h}e\left(f(j)\frac{d}{h}\right)\right|L+O(h)\le L\left(1-\frac{c_1}{h^4}\right)+O\left(h\right),
\end{equation}
holds for some absolute constant $c_1>0$. Moreover, from Lemma \ref{lem42} there exists a constant $C_f\in\zb_+$ depending only on $f$ such that if $h>C_f$ then
\begin{equation}\label{eq412}
\left|\sum_{1\le n\le L}e(f(n) y)\right|\le L^{1-2^{-r-1}}+Lh^{-2^{-r-1}}.
\end{equation}
The use of \eqref{eq411} and \eqref{eq412} yields there exists a constant $c_f\in(0,1)$ depending only on $f$ such that
\begin{equation*}
\left|\sum_{1\le n\le L}e(f(n) y)\right|\le (1-c_f)L
\end{equation*}
holds for all sufficiently large $L$. Thus we obtain that,
\begin{align}\label{eq413}
{\rm Re}\sum_{1\le n\le L}\left[1-e\left(\frac{j}{k}-f(n)y\right)\right]\ge L-\left|\sum_{1\le n\le L}e\left(f(n)y\right)\right|\ge c_fL\gg L.
\end{align}

Combining \eqref{eq400} and \eqref{eq413}, the proof follows.
\end{proof}

We now give the proof of Proposition \ref{pro2}. Let $\delta\mid\Pi_f$ and $k\in\zb_+$ such that for any $p|k$ we have $p\delta\nmid \Pi_f$.  From Proposition \ref{main1} and for all $x>0$ we have
\begin{align*}
E_{f,k,\delta,a}(x):=&\sum_{\substack{n\ge 0\\ n\equiv af(0)~(\bmod \delta)}}\left|p_{f}(a,\delta k;n)-\frac{p_f( n)}{k}\right|^2e^{-2nx}\\
=&\int_{-1/2}^{1/2}\left|\sum_{\substack{n\ge 0\\ n\equiv af(0)~(\bmod \delta)}}\left(p_{f}(a,\delta k;n)-\frac{p_f( n)}{k}\right)e^{-nx}e(-ny)\right|^2\,dy\\
=&\int_{-1/2}^{1/2}\left|\frac{1}{k\delta}\sum_{1\le j<k, 0\le \ell<\delta}\zeta_{\delta k}^{-ja-k\ell af(0)}G_f\left(\zeta_{\delta k}^{j}, \zeta_{\delta }^{\ell} e^{-x-2\pi\ri y}\right)\right|^2\,dy\\
\le &G_f(x)^2\int_{-1/2}^{1/2}\max_{\substack{1\le j<k\\ 0\le \ell<\delta}}\exp\left(-F_{f,k,\delta, j,\ell}(x,y)\right)\,dy,
\end{align*}
where
\begin{align*}
F_{f,k,\delta, j,\ell}(x,y)&=-\log \left|\frac{G_f\left(\zeta_{\delta k}^{j}, \zeta_{\delta }^{\ell} e^{-x-2\pi\ri y}\right)}{G_f(x)}\right|^2\nonumber\\
&=-\log \left|\prod_{n\ge 1}\frac{1-e^{-f(n)x}}{1- e^{-f(n)x}e(\frac{j+k\ell}{\delta k}-f(n)y)}\right|^2\nonumber\\
&=2\sum_{n,c\ge 1}\frac{e^{-f(n)c x}}{c}{\rm Re}\left[1-e\left(c\left(\frac{j+k\ell}{\delta k}-f(n)y\right)\right)\right].
\end{align*}
Furthermore,
\begin{align*}
F_{f,k,\delta, j,\ell}(x,y)&\ge 2\sum_{n\ge 1}e^{-f(n) x}{\rm Re}\left[1-e\left(\frac{j+k\ell}{\delta k}-f(n)y\right)\right]\\
&\ge \frac{2}{e^{xf(x^{-1/r})}} {\rm Re}\sum_{1\le n\le x^{-1/r}}\left[1-e\left(\frac{j+k\ell}{\delta k}-f(n)y\right)\right]\\
&\gg {\rm Re}\sum_{1\le n\le x^{-1/r}}\left[1-e\left(\frac{j+k\ell}{\delta k}-f(n)y\right)\right],
\end{align*}
and the use of Lemma \ref{Z} implies that there exists a constant $\delta_f'\in\rb_+$ depending only on $f$ such that
\begin{align*}
\min_{\substack{-1/2\le y\le 1/2\\ 1\le j<k, 0\le \ell<\delta}}F_{f,k,\delta, j,\ell}(x,y)\ge \begin{cases}\qquad 0 & if~k^{-2}x^{-1/r}= O(1),\\
 2\delta_f' k^{-2}x^{-1/r} &if~k^{-2}x^{-1/r}\rrw +\infty.
 \end{cases}
\end{align*}
Therefore,
\begin{align*}
E_{f,k,\delta,a}(x)&\ll G_f(x)^2\exp\left(-\min_{\substack{-1/2\le y\le 1/2\\ 1\le j<k, 0\le \ell<\delta}}F_{f,k,\delta, j,\ell}(x,y)\right)\\
&\ll G_f(x)^2\exp\left(-2\delta_f'k^{-2}x^{-1/r}\right),
\end{align*}
as $x\rrw 0^+$. This completes the proof of Proposition \ref{pro2}.

%\bibliographystyle{plain}
%\bibliography{test190504}

\paragraph{Acknowledgements.} The author would like to thank the referee for very helpful
and detailed comments and suggestions. This research was partly supported by the
National Science Foundation of China (Grant No. 11971173).

\bigskip
\noindent
{\sc Nian Hong Zhou\\
School of Mathematics and Statistics, Guangxi Normal University\\
No.1 Yanzhong Road, Yanshan District, Guilin, 541006\\
Guangxi, PR China}

\noindent and

\noindent
{\sc School of Mathematical Sciences, East China Normal University\\
500 Dongchuan Road, Minhang District, 200241\\
Shanghai, PR China}\newline
Email:~\href{mailto:nianhongzhou@outlook.com; nianhongzhou@gxnu.edu.cn}{\small nianhongzhou@outlook.com; nianhongzhou@gxnu.edu.cn}

\end{document}